\documentclass[letterpaper,11pt]{amsart}
\usepackage{tikz}
\usepackage{wrapfig}
\usetikzlibrary{arrows}
\usepackage{float}
\usepackage{subcaption} 
\usepackage{amsmath}
\usepackage[font=footnotesize,labelfont=bf]{caption}
\usepackage{latexsym,array,delarray,amsthm,amssymb,epsfig,setspace}

\usepackage{mathtools} 
\DeclarePairedDelimiter{\floor}{\lfloor}{\rfloor}

\theoremstyle{plain}
\newtheorem{thm}{Theorem}[section]
\newtheorem{lemma}[thm]{Lemma}
\newtheorem{prop}[thm]{Proposition}
\newtheorem{cor}[thm]{Corollary}

\newtheorem*{thm*}{Theorem}
\newtheorem*{lemma*}{Lemma}
\newtheorem*{prop*}{Proposition}
\newtheorem*{cor*}{Corollary}
\newtheorem*{conj*}{Conjecture}

\theoremstyle{definition}
\newtheorem{defn}[thm]{Definition}
\newtheorem{ex}[thm]{Example}

\theoremstyle{remark}

\newcommand{\rr}{\mathbb{R}}

\newcommand{\conv}{\mathrm{conv}}
\newcommand{\ind}{\mbox{$\perp \kern-5.5pt \perp$}}

\newcommand{\shape}{\mathrm{shape}}

\title{Exchangeable and Sampling Consistent Distributions on Rooted Binary Trees}
\author{Benjamin Hollering and Seth Sullivant}

\begin{document}
\maketitle

\begin{abstract}
We introduce a notion of finite sampling consistency for phylogenetic trees and 
show that the set of finitely sampling consistent and exchangeable distributions 
on $n$ leaf phylogenetic trees is a polytope. We use this polytope to show that 
the set of all exchangeable and infinite sampling consistent distributions on 4 leaf phylogenetic 
trees is exactly Aldous' beta-splitting model and give a description of some of 
the vertices for the polytope of distributions on 5 leaves. We also introduce a 
new semialgebraic set of exchangeable and sampling consistent models we call the multinomial model and use it to characterize the set of exchangeable and sampling consistent distributions. 
\end{abstract}


\section{Introduction}
Leaf-labelled binary trees, which are commonly called phylogenetic trees, 
are frequently used to represent the evolutionary relationships between species. 
In this paper we will restrict our attention to rooted binary trees and our label set 
for a tree with $n$ leaves will always be $[n] = \{1,2, \ldots n\}$ and  call such 
trees $[n]$-trees, the set of which we denote $RB_L(n)$. 

Processes for generating random $[n]$-trees play an important role in 
phylogenetics.  Two common examples are the uniform distribution (where
a tree is chosen uniformly at random from among all trees in $RB_L(n)$)
and the Yule-Harding distribution (a simple Markov branching process).
Some other examples of random tree models include Aldous' $\beta$-splitting 
model \cite{al93}, the $\alpha$-splitting model \cite{ford06}, 
and the coalescent process (which generates trees with edge lengths) 
\cite{wakeley2008}.
Two features common to all these random tree processes and desirable
for any such tree process is that they are
exchangeable and sampling consistent.  

Let $p_n$ denote a probability distribution on $RB_L(n)$.
\emph{Exchangeability} refers to the fact that relabeling the leaves of the
tree does not change its probability. That is, for all $T \in RB_L(n)$ and $\sigma \in S_n$,
$p_n(T) = p_n(\sigma T)$.  Exchangeability is a natural condition since it does not
allow the names of the species to play any special role in the probability distribution.
A family of distributions, $\{p_n\}_{n=2}^\infty$, on trees has 
\emph{sampling consistency} if for each $n$, the distribution $p_n$, 
which is on $[n]$-trees, can be realized as the marginalization of distributions $p_m$, 
which is on $[m]$-trees, for $m > n$. That is the probability of 
a $[n]$-tree, $T$, under $p_n$ can be written as 
\[
\pi_n(p_m) (T) = p_n^m(T) = \sum_{\{S \in RB_L(m) | T = S|_{[n]}\}} p_m(S).
\]
Sampling consistency is a natural condition for a random tree model because it means that randomly missing species do not affect the underlying distribution on the species that were observed.  

The goal of this paper is to study the structure of finitely sampling consistent distributions on rooted binary trees. In particular, we aim to obtain a finite deFinetti-type theorem for these trees in the style of Diaconis' Theorem 4 in \cite{di97}. 
Our motivation is two-fold.  First of all,  
there has been significant work on understanding the set of exchangeable, 
sampling consistent distributions on other discrete objects, including rooted trees. A classic result 
in this theory is deFinetti's Theorem for infinitely exchangeable sequences of 
binary random variables which shows that every subsequence of the 
infinite sequence can be expressed as a mixture of independent and identically distributed sequences. This does not hold for finitely exchangeable sequences but
Diaconis later developed a finite form of deFinetti's theorem. He showed that 
if a finite exchangeable sequence of binary random variables, $\{X_i\}_{i=1}^{n}$, 
can be extended to an exchangeable sequence, $\{X_i\}_{i=1}^{m}$ where $m > n$, 
then the original sequence can be approximated with a mixture of independent and 
identically distributed sequences with error $O(\frac{1}{m})$ \cite{di97}. A substantial amount of work has been done on exchangeable arrays (see \cite{diaconisJanson} for example) as well, which has been used to prove deFinetti theorems for other discrete objects. For instance, Lauritzen, Rinaldo, and Sadeghi recently developed a deFinetti Theorem for exchangeable random networks \cite{la17}. 

As previously mentioned, there has already been considerable work characterizing exchangeable and sampling consistent distributions on trees using weighted real trees as limit objects in \cite{forman18,formanPitman18,winkel08}. In \cite{winkel08} a characterization of the exchangeable and sampling consistent Markov branching models we discuss in Section \ref{sec:splitModels} is obtained. A true deFinetti theorem for trees is conjectured in \cite{formanPitman18} and proven in Theorem 3 of \cite{forman18}. The approach taken in these papers is to characterize all infinitely sampling consistent distributions on trees using a tree-limit like object called a weighted real tree. In this paper, we instead take a geometric and combinatorial approach to the study of exchangeable and finitely sampling consistent distributions on binary trees and examine what happens as we take the limit. 

A second motivation comes from the combinatorial phylogenetics problem
of studying properties of the distribution of the maximum agreement 
subtree of pairs of random trees.
Let $T \in   RB_L(n)$ and $S \subseteq [n]$.  The restriction tree $T|_S$
is the rooted binary tree with leaf label set $S$
obtained by removing all leaves of $T$ not in $S$ and suppressing all 
vertices of degree 2 except the root.
Two trees, $T_1, T_2 \in RB_L(n)$, agree on a set 
$S \subseteq [n]$ if $T_1|_S = T_2|_S$. A maximum agreement set is an agreement set of the
largest size for $T_1$ and $T_2$.  
The size of a maximum agreement subtree of these two trees is the 
cardinality of the largest subset $S$ that $T_1$ and $T_2$ agree on and is denoted $MAST(T_1,T_2)$. 
If $S$ is an agreement set with $|S| = MAST(T_1,T_2)$ then 
the resulting tree $T_1|_S = T_2|_S$ is a \emph{maximum agreement subtree} of $T_1$ and $T_2$.   

Understanding the distribution of $MAST(T_1, T_2)$ for random tree distributions
would help in conducting hypothesis tests that the similarity between the trees 
is no greater than the similarity between random trees. For example, it was 
suggested in \cite{devi07} that $MAST(T_1, T_2)$ could be used to test the 
hypothesis that no cospeciation occurred between a family of host species 
and a family of parasite species that prey on them. The study of the 
distribution of $MAST(T_1,T_2)$ for random trees $T_1,T_2$ is primarily 
conducted with the assumption that $T_1$ and $T_2$ are drawn from an exchangeable, 
sampling consistent distribution on rooted binary trees. 
Bryant, Mackenzie, and Steel began the study of the distribution of $MAST(T_1,T_2)$ 
and obtained some first bounds on $\mathbb{E}(MAST(T_1,T_2))$ for random trees $T_1$ and $T_2$ drawn from the Uniform or Yule-Harding distributions \cite{br03}. 
Later work on the distribution obtained an upper bound on the order of 
$O(\sqrt{n})$ for $\mathbb{E}(MAST(T_1,T_2))$ when $T_1$ and $T_2$ are drawn from any exchangeable, 
sampling consistent distribution \cite{be15}. A lower bound on the order of $\Omega(\sqrt{n})$ 
has been conjectured for all exchangeable, sampling consistent distributions as well 
but this remains an open problem. Our hope in pursuing this project is that 
developing a better understanding of the set of all exchangeable sampling
consistent distributions might shed light on this conjecture.

In this paper we study the structure of exchangeable, 
sampling consistent distributions on leaf labelled, rooted binary trees. 
We introduce a notion of a polytope of exchangeable and finitely sampling consistent distributions. 
We use it to study the set of exchangeable and sampling consistent distributions 
on trees and get some characterizations for trees with a small number of leaves. 
We show that set of all exchangeable and sampling consistent distributions 
on four leaf trees come from the $\beta$-splitting model that was first introduced by Aldous in \cite{al93}. 
We have not been able to find a similar characterization for exchangeable and 
sampling consistent distributions on five leaf trees but we describe some of 
the vertices of the polytope of exchangeable and finitely sampling consistent 
distributions. We also introduce a new exchangeable and sampling consistent 
model on trees, called the multinomial model, and show that 
every sampling consistent and exchangeable distribution can be 
realized as a convex combination of limits of sequences of multinomial distributions. 


\section{Exchangeability and Finite Sampling Consistency}
In this section we describe how the set of exchangeable distributions 
relates to the set of all distributions on leaf labelled, rooted binary trees. 
We then introduce a notion of finite sampling consistency and discuss 
how it relates to traditional sampling consistency.

Recall that $RB_L(n)$ denotes the set of all leaf labelled, 
rooted binary trees with label set $[n]$, which we call $[n]$-trees, and that $|RB_L(n)| = (2n-3)!!$. 
The set of all distributions on $RB_L(n)$ is the probability 
simplex $\Delta_{(2n-3)!!-1} \subseteq \mathbb{R}^{(2n-3)!!}$ 
where the coordinates are indexed by $[n]$-trees. 
The symmetric group $S_n$ denotes the group of permutations of $[n]$.  
For each $\sigma \in S_n$ and $T \in RB_L(n)$ let $\sigma T$ denote
the tree obtained by applying $\sigma$ to the leaf labels.

\begin{defn}
A distribution $p$ on $RB_L(n)$ is \emph{exchangeable} 
if for all permutations $\sigma \in S_n$ and $[n]$-trees $T \in RB_L(n)$, 
$p(T) = p(\sigma T)$. The set of all exchangeable distributions on $RB_L(n)$ is
denoted $EX_n$. 
\end{defn}

As previously mentioned, exchangeability  requires that the 
probability of a $[n]$-tree under a particular distribution depend only on 
the shape of the tree. Thus we
only need to consider distributions on the set of tree shapes. 
Let $RB_U(n)$ 
denote the set of unlabelled rooted binary trees, which we may also 
call trees or tree shapes. This idea is summarized in the next 
lemma which is  the $[n]$-tree analogue of Lemma 2 in \cite{la17}. 

\begin{lemma}\label{lemma:exnsimplex}
The set of exchangeable distributions on $RB_L(n)$,  $EX_n$, is a simplex of dimension 
$|RB_U(n)|-1$ with coordinates indexed by tree shapes.
\end{lemma}

\begin{proof}
First we define a distribution $p_{T} \in EX_n$ for each tree 
shape $T \in RB_U(n)$. To do so, we let $O(T)$ be the set of 
trees $T' \in RB_L(n)$ such that $\shape(T') = T$. For any tree $S \in RB_L(n)$ we set  
\[
p_T(S) = \begin{cases}
            \frac{1}{|O(T)|} & shape(S) = T \\
            0 & shape(S) \neq T.
         \end{cases}
\]
Then $p_T \in EX_n$ since it is a probability distribution 
on trees and all trees of the same shape have the same probability. 
We claim that $EX_n = \conv\left( \{p_T : T \in RB_U(n) \}  \right)$, where $\conv(A)$ 
denotes the convex hull of the set $A$. 
Since $p_T \in EX_n$ for all $T \in RB_U(n)$, it is enough to show that 
any distribution $p \in EX_n$ can be written as a convex combination of the $p_T$. 
If $p \in EX_n$, then the probability of any tree $T' \in RB_L(n)$ depends 
only on the shape of $T'$ not the leaf labelling so we can write
\[
p = \sum_{T \in RB_U(n)} p(T) \cdot |O(T)| \cdot p_T
\]
where $p(T)$ represents the probability of any $[n]$-tree in $RB_L(n)$ with shape $T$. 
Since the original $p$ is a probability distribution on all leaf labelled trees
the weights in the linear combination are nonnegative and sum to $1$. 

Lastly we note that the vectors $p_T$ are affinely independent since there 
is no overlap of coordinate indices where the entries in $p_T$ are nonzero. 
So $EX_n = \conv\left( \{p_T : T \in RB_U(n) \}  \right)$
 is a simplex and has coordinates indexed by $RB_U(n)$. 
\end{proof}

Lemma \ref{lemma:exnsimplex} allows us to move from studying exchangeable 
distributions on leaf labelled $[n]$-trees to all distributions on unlabelled trees. 
We will primarily focus on understanding the set of sampling consistent distributions 
within $EX_n$ now. First recall that for $p_m \in EX_m$ the marginalization or 
projection map $\pi_n$, gives a new distribution 
$p_n^m$ on $RB_L(n)$ for $n < m$, defined for all $T \in RB_L(n)$ by
\[
\pi_n(p_m)(T) = \sum_{\{S \in RB_L(m) | T = S|_{[n]}\}} p_m(S) 
\]
We will use this marginalization map to define a notion of finite sampling consistency.

\begin{defn}
A family of distributions $\{p_k\}_{k=n}^m$ is \emph{finitely sampling consistent} 
or \emph{$m$-sampling consistent}, if for each $n \leq k < m$, $p_k = \pi_k(p_m)$. 
We denote the set of all distributions in  $EX_n$ that are $m$-sampling consistent
by 
\[
EX_n^m = \pi_n(EX_m).
\]
\end{defn}

It is immediate that if a distribution in $EX_n$ is $m$-sampling consistent, 
then for any $k$, such that $n < k < m$, the distribution is also $k$-sampling consistent.
This leads to the following:

\begin{lemma}\label{lem:monotone}
For all $m > k > n$, 
\[
EX_n^m \subseteq EX_n^k.
\]
\end{lemma}

A distribution in $EX_n$, is sampling consistent if it is part of a $m$-sampling 
consistent family of distributions for all $m>n$. In other words, 
a distribution is sampling consistent if it is in $EX_n^m$ for all $m>n$. 
Thus we can define the following notation for the set of 
exchangeable distributions on $RB_L(n)$ that are sampling consistent:
\[
EX_n^\infty  :=  \cap_{m  = n}^\infty  EX_n^m.
\]

\begin{lemma}\label{lemma:exnm}
Let $p_T \in EX_m$ be defined as it is in Lemma \ref{lemma:exnsimplex}, 
then  
\[
EX_n^m = \conv\left(\{\pi_n(p_T) : T \in RB_U(m)  \} \right).
\]
\end{lemma}

\begin{proof}
Clearly it holds that $\conv(\{\pi_n(p_T)  : T \in RB_U(m)  \}) \subseteq EX_n^m$ since
$\pi_n(p_T) \in EX_n^m$ for all $T \in RB_U(m)$. It is enough to show that
if we have a distribution $p_n^m \in EX_n^m$, then it can be written 
as a convex combination of the $\pi_n(p_T)$.  
If $p_n^m \in EX_n^m$, then there exists $p_m \in EX_m$ such that
$\pi_n(p_m) = p_n^m$. Since $p_m \in EX_m$, we know from Lemma 2.2 that we can write
$p_m = \sum_{T \in RB_U(n)} p_m(T) \cdot |O(T)| \cdot p_T$. 
Then evaluating $\pi_n(p_m)$ at a $[n]$-tree $S \in RB_L(n)$ gives 
\[
\pi_n(p_m)(S) = \sum_{\{Q \in RB_L(m) | S = Q|_{[n]}\}} \sum_{T \in RB_U(m)} 
p_m(T) \cdot |O(T)| \cdot p_T(Q)
\]
Changing the order of summation we have
\[
\pi_n(p_m)(S) = \sum_{T \in RB_U(m)} p_m(T) \cdot |O(T)| \sum_{\{Q \in RB_L(m) | S = Q|_{[n]}\}} p_T(Q)
\]
but $\sum_{\{Q \in RB_L(m) | S = Q|_{[n]}\}} p_T(Q) = \pi_n(p_T)(S)$ so we get that 
\[
\pi_n(p_m)(S) = \sum_{T \in RB_U(m)} p_m(T) \cdot |O(T)| \cdot \pi_n(p_T) (S)
\]
which shows that $p_n^m = \pi_n(p_m)$ can be written 
as a convex combination of the $\pi_n(p_T)$. 
\end{proof}

\begin{ex}
While it will be the case that $EX_n^m = \conv(\{\pi_n(p_T) : T \in RB_U(m) \})$, 
not every  $\pi_n(p_T)$ will be a vertex of $EX_n^m$. Figure \ref{fig:ex57} illustrates this. 
\end{ex}

\begin{figure}
    \centering
    \resizebox{!}{2in}{\includegraphics{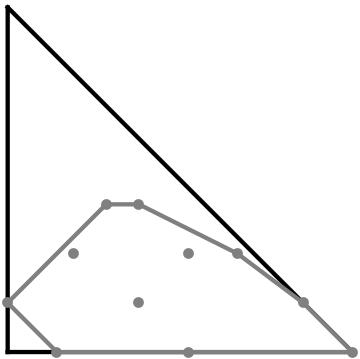}}
    \caption{The projection of $EX_5^7$ onto the first two coordinates of the simplex $EX_5$. The gray points correspond to the points $\pi_n(p_T)$ for $T \in RB_U(7)$.}
    \label{fig:ex57}
\end{figure}

Lemma \ref{lemma:exnm} implies that understanding how 
the marginalization map acts on the vertices of $EX_m$ will allow us to 
compute all of $EX_n^m$. The following lemma and corollary will give 
us a method for calculating the vertices of $EX_n^m$ by computing subtree densities.

\begin{lemma}
Let $S \in RB_L(n)$ and $T \in RB_U(m)$. 
Also let $c_T(S) = |\{Q \in RB_L(m) | S = Q|_{[n]}, shape(Q) = T \}|$. 
Then $\pi_n(p_T)(S) = \frac{c_T(S)}{|O(T)|}$. 
\end{lemma}

\begin{proof}
By definition of the map $\pi_n$
\[
\pi_n(p_T)(S) = \sum_{\{Q \in RB_L(m) | S = Q|_{[n]}\}} p_T(Q)
\]
but $p_T(Q)$ is nonzero if and only if $shape(Q) = T$, in which case it is $\frac{1}{|O(T)|}$. 
So the above sum becomes
\[
\pi_n(p_T)(S) = \sum_{\{Q \in RB_L(m) | S = Q|_{[n]}, shape(Q) = T \}} \frac{1}{|O(T)|}
= \frac{c_T(S) }{|O(T)|}.
\]
\end{proof}

\begin{cor}\label{cor:induceddensity}
Let $S' \in RB_U(n)$ and $T \in RB_U(m)$. Then
$\pi_n(p_T)(S')$, which is used to denote the sum of $\pi_n(p_T)(S
)$ over all $S \in O(S')$, is the induced subtree density of $S'$ in $T$. That is, 
$\pi_n(p_T)(S')$ is the ratio of the number of times that $S'$ occurs as a restriction tree of $T$ when $n-m$ of its leaves are marginalized out. 
\end{cor}
\begin{proof}
From the previous lemma, we know that for any $S \in O(S')$, $\pi_n(p_T)(S) = \frac{c_T(S)}{|O(T)|}$ where $c_T(S) = |\{Q \in RB_L(m) | S = Q|_{[n]}, shape(Q) = T \}|$. Then we have
\[
\pi_n(p_T)(S') = \sum_{S \in O(S')} \frac{c_T(S)}{|O(T)|}
\]
So for each labelling $S$ of $S'$, we are counting which fraction of labellings of $T$ yield $S$ when restricted to $[n]$. As we sum over all labellings of $S$, this gives us the total fraction of times that the shape $S'$ appears as a restriction tree of the shape $T$ when $(n-m)$ of its leaves are marginalized out.
\end{proof}

The following examples elucidates what is meant by induced subtree density and shows how we can explicitly calculate this quantity.

\begin{ex}
We show how to find the projection of one vertex of $EX_5$ down to $EX_4$. 
$EX_4^5$ is the convex hull of the projection of all of the vertices of $EX_5$. 
Begin with the tree shape $T$ pictured in Figure \ref{fig:5leaf}.
We label the leaves of $T$ for the sake of the calculation but it should 
be thought of as an unlabelled tree. We then find the shape of the restriction
tree for the five $4$-subsets of $[5]$. The restriction of $T$ to the leaf sets 
$\{1,2,3,4\}, \{1,2,3,5\}, \{1,2,4,5\}$, gives the shape $Comb_4$ 
 and the restriction to the sets $\{1,3,4,5\}, \{2,3,4,5\}$ gives the shape $Bal_4$, 
 pictured in Figure \ref{fig:2-4leaf}. We let the first coordinate of $EX_4$ be the 
 probability of obtaining $Comb_4$ and the second be the probability of 
 obtaining $Bal_4$. As mentioned above, these probabilities 
 will simply be the number of times each shape appears as a restriction tree over the total number of restriction trees. Thus this vertex of $EX_5$ will give us the distribution $(2/5,3/5)$ in $EX_4$. 
\end{ex}

\begin{figure}
    \centering
    \begin{subfigure}[b]{0.3\linewidth}
        \centering
        \begin{tikzpicture}[scale=.5,thick]
            \draw (2.5,5)--(0,0);
            \draw (2.5,5)--(5,0);
            \draw (.5,1)--(1,0);
            \draw (4.5,1)--(4,0);
            \draw (4,2)--(3,0);
            \draw (0,0) node[below]{$1$};
            \draw (1,0) node[below]{$2$};
            \draw (3,0) node[below]{$3$};
            \draw (4,0) node[below]{$4$};
            \draw (5,0) node[below]{$5$};
        \end{tikzpicture} 
        \caption{A five leaf tree}
        \label{fig:5leaf}
    \end{subfigure}
    \begin{subfigure}[b]{0.3\linewidth}
        \centering
        \begin{tikzpicture} [scale = .5,thick]
            \draw (2.5,3)--(4,0);
            \draw (2.5,3)--(1,0);
            \draw (3,2)--(2,0);
            \draw (3.5,1)--(3,0);
            \draw (2.5,-.5) node[below]{$Comb_4$};
            \draw (7,3)--(5,0);
            \draw (7,3)--(9,0);
            \draw (5.66,1)--(6.33,0);
            \draw (8.33,1)--(7.66,0);
            \draw (7,-.5) node[below]{$Bal_4$};
        \end{tikzpicture}
        \caption{The two tree shapes for four leaf trees}
        \label{fig:2-4leaf}
    \end{subfigure}
    \caption{}
    \label{fig:my_label}
\end{figure}

We have now seen how to compute the vertices of $EX_n^m$ 
explicitly but not every distribution $\pi_n(p_T)$ is a vertex
of  $EX_n^m$.
However, the comb tree always yields a vertex of $EX_n^m$.  

\begin{lemma}\label{lem:combvert}
For all $m \geq n$, let $Comb_m \in RB_U(m)$ be the $m$-leaf comb tree, 
then $p_{Comb_m}$ is a vertex in $EX_n^m$. 
\end{lemma}

\begin{proof}
The comb tree has only smaller comb trees as restriction trees, 
so the image of the comb distribution on $m$ leaves under the marginalization map will 
be the comb distribution on $n$ leaves.
Since $p_{Comb_n}$ is a vertex of $EX_n$ and $EX_n^m$ is a subset of $EX_n$,
then $p_{Comb_n}$ is also a vertex of $EX_n^m$.
\end{proof}


\section{Examples of Exchangeable and Sampling Consistent distributions}

In this section we discuss some of the well-known exchangeable and 
sampling consistent families of distributions 
particularly, the Markov branching models.  We also introduce a new family of exchangeable sampling
consistent tree distributions, namely the multinomial family.

\subsection{Markov Branching Models}
\label{sec:splitModels}
An important example of sampling consistent and exchangeable distributions 
are the families of Markov branching models which can be constructed in the 
following way as first introduced in \cite{al93} by Aldous. 

Suppose that for every integer $n \geq 2 $, we have a probability distribution on $\{1,2, \ldots, n-1\}$
$q_n = (q_n(i) : i = 1,2, \ldots n-1)$ which satisfies $q_n(i) = q_n(n-i)$.
Using this family of distributions we can define a probability distribution 
on $RB_U(n)$ by taking the probability that $i$ leaves fall on the left 
of the root-split and $n-i$ leaves fall on the right of the root-split to be 
$q_n(i)$ with each choice of $i$ labels to fall on the left having the same probability. 
Repeating recursively in each branch will yield the probability of a rooted binary tree. 
Aldous called these models Markov branching models. 

Haas et al. classified the sampling consistent Markov branching models on rooted binary trees in \cite{winkel08}. They show that every sampling consistent Markov branching model, defined by the splitting rules $q_n$, $n \geq 2$, has an integral representation of the form
\begin{equation}\label{eq:qni}
q_n(i) = a_n^{-1} \bigg(\binom{n}{i}\int_{0}^{1}x^i (1-x)^{n-i} \nu(dx) + nc1_{i=1} \bigg)
\end{equation}
where $c \geq 0$, $\nu$ is a symmetric measure on $(0,1)$ such that $\int_{0}^{1}x (1-x) \nu(dx) < \infty$, and $a_n$ is a normalization constant. $c1_{i=1}$ accounts for the comb distribution. 
A subclass of these models are those where the measure $\nu$ in equation (\ref{eq:qni}) has the form
$\nu(dx) = f(x)dx$ for a probability density function $f$ on $(0,1)$ that is symmetric on the interval (i.e.~$f(x) = f(1-x)$) and where $c=0$. 
These Markov branching models can be thought of as uniformly choosing $n$ points in the interval $(0,1)$ at random and then splitting the interval with respect to the density $f$. 
Repeating the splitting process recursively in each subinterval until 
each of the original $n$ points is contained in its own subinterval gives a tree shape. This process is pictured in Figure 6 in \cite{al93}.

One particularly important family of Markov branching distributions is the 
beta-splitting model. 
It is a Markov branching model that belongs to the subclass mentioned above where the function $f$ in the above description has the form
\[
f(x) = \frac{\Gamma(2\beta +2)}{\Gamma^{2}(\beta+1)}x^{\beta}(1-x)^{\beta}
\]
for  $-1 < \beta < \infty$.  
For the beta-splitting model we can calculate the values $q_n(i)$ explicitly in terms of $\beta$. 
By plugging in the beta-splitting density function $f$ into (\ref{eq:qni})
for $q_n(i)$  we get the following formulas: 
\begin{equation}\label{eq:beta}
q_n(i) = a_n^{-1} \binom{n}{i} \frac{\Gamma(\beta + i + 1)\Gamma(\beta + n -i +1)\Gamma(2\beta+2)}{\Gamma(\beta+n+2)\Gamma^{2}(\beta+1)}
\end{equation}
for $-1 < \beta < \infty$.  Note that (\ref{eq:beta}) gives a valid probability distribution
when $-2 < \beta \leq -1$ and so it is natural to extend the beta-splitting model
to those values of $\beta$, although the density is not well-defined in that case.
As $\beta$ approaches $-2$ the beta-splitting model approaches the distribution
which puts all probability on the comb tree, so we also include $\beta = -2$  in the beta
splitting model as the comb distribution. 

An important note here is that for the beta-splitting model 
each $q_n(i)$ is actually a rational function in $\beta$. 
Using properties of the gamma function one can see that the above formula simplifies to
\[
q_n(i) = \frac{\binom{n}{i}(i+\beta)_i (n-i+\beta)_{n-i}}{(n+2\beta+1)_n - 2(n+\beta)_n}
\]
Since each $q_n(i)$ is a rational function in $\beta$, we can see that the probability 
of obtaining a certain tree shape is a rational function in $\beta$ as well 
because the probability of obtaining that tree shape under the beta-splitting 
model is simply the product of the probability of all of the splits in the tree. 
\begin{ex}
Let $Comb_4$ and $Bal_4$ be the trees pictured in Figure \ref{fig:2-4leaf}. Then the probabilities of obtaining them under the beta-splitting model are
\begin{align*}
    p(Comb_4) & = 2q_4(1)=\frac{12+4\beta}{18+7\beta} \\
    p(Bal_4) & = q_4(2)=\frac{6 + 3\beta}{18+7\beta}
\end{align*}

\end{ex}
This model also has a nice characterization among all of the sampling consistent Markov branching models. In \cite{gibbsTrees}, Mccullagh, Pitman, and Winkel show that the beta-splitting models are the only sampling consistent Markov branching models whose splitting rules admit a particular factorization.  

We are interested in examining how the sampling consistent Markov branching models and in particular the beta-splitting model fits inside inside of $EX_n$ as a whole.  These distributions are infinitely sampling consistent
and so lie in $EX_n^\infty$ as well. 
A priori, it might seem that to determine the probability of a tree 
shape with $n$ leaves under a Markov branching model that one would need to have 
not only the distribution $q_n$ but also distributions $q_k$ where $2 \leq k \leq n-1$. 
This is actually not the case for any sampling consistent Markov branching model though. Ford showed in Proposition 41 of \cite{ford06} that if $(q_k | 2 \leq k \leq n)$ are the splitting rules for a distribution in $EX_n^\infty$, then in fact it must be that
\begin{equation}\label{eq:qnrecurrence}
q_{n-1}(i) = \frac{(n-i)q_n(i) + (i+1)q_n(i+1)}{n-2q_n(1)}
\end{equation}
This implies that all that is needed to define a distribution in $EX_n^\infty$ is the first splitting rule $q_n$ which gives the following corollary. 

\begin{cor}
The dimension of the set of all sampling consistent Markov branching models in $EX_n$ is at most $\lceil {\frac{n-1}{2}} \rceil -1$
\end{cor}

\begin{proof}
As explained above, a Markov branching model is completely determined by the distribution 
$q_n = (q_n(i) : i = 1,2, \ldots n-1)$ which determines all of the distributions 
$q_k = (q_k(i) : i = 1,2, \ldots k-1)$ where $ 2 \leq k \leq n-2$. 
Since $q_n$ must be symmetric we immediately get that the values 
$q_1,q_2, \ldots, q_{\lceil {\frac{n-1}{2}} \rceil}$ determine all of $q_n$. 
Also since $q_n$ must be a distribution we lose one of these as a free parameter, thus 
the dimension of the set of sampling consistent Markov branching models is bounded above by $(\lceil {\frac{n-1}{2}} \rceil -1)$.
\end{proof}

Note that when $n = 4$, the 
space of sampling consistent Markov branching models has dimension $1$. We will see in Section \ref{sec:ex4}
that the set of beta-splitting
models is equal to the set of sampling consistent Markov branching models in this case.


\subsection{Multinomial model}
\label{sec:multimodel}
The multinomial model is a model that associated to each tree shape
 $T \in RB_U(m)$ for any $ m \geq 2$
 a family of probability distributions on $RB_L(n)$ for each $n$.
We will often extend the model to allow to use extended trees with
an additional leaf added to the root.
We associate to every edge, $e$, in $T$ a parameter $t_e \geq 0$. 
This gives us a vector of parameters $t = (t_e | e \in E(T))$ of length $2m-1$,
and we assume that $\sum_e t_e = 1$, so that these parameters give a probability
distribution on the edges of $T$.
We will now use this probability distribution to define a set of distributions on 
$RB_U(n)$ for any $n \geq 2$. Note that $n$ and $m$ do not have to be related to each other.

Using the distribution $t$, we draw a multiset $A$ of edges from the tree $T$,
where edge $e$ occurs with probability $t_e$.  
There is a natural way to take the tree $\tilde{T}$ and a 
multiset $A$ of size $n$ on the set of parameters and construct a 
new tree which we will call $T_A \in RB_U(n)$. Each time that an edge $e$ 
appears in $A$, we add a new leaf to the edge $e$, which will give us a new tree 
with an undetermined number of leaves. We then simply take $T_A$ to be the induced 
subtree on only the leaves that come from $A$.   Hence, the multinomial model on the
tree $T$ gives a way to produce random trees with an underlying skeleton that is the tree $T$.
For large $n$, the resulting random trees look like $T$ with many extra leaves added.

The multinomial probability of observing a particular multiset of edges $A$
is the monomial
\[
p_A = \binom{n}{m_A}
\prod_{e \in T} t_e^{m_A(e)}
\]
where $m_A(e)$ denotes the number of times that $e$ appears in the multiset $A$,
and $m_A$ is the resulting vector.  

Letting $M_n^T$ be the set of all $n$ element multisets of edges of $T$, 
we can calculate the probability of observing any particular tree shape $S$ by 
\[
p_{T,t}(S) = \sum_{\substack{A \in M_n^T \\ T_A = S}} p_A.
\]

\begin{ex}
Consider the tree $T$  from Figure \ref{fig:tripod}
 with edge parameters $(t_1,t_2,t_3)$. To calculate the probability of the tree, 
 $Bal_5$, in Figure \ref{fig:balance} we use the formula
\[
p_{T,t}(Bal_5) = \sum_{\substack{A \in M_5^T \\ T_A = S}} p_A.
\]
The only multisets that satisfy this condition are the sets 
$A_1 = \{2,2,2,3,3\}$ and $A_2 =  \{2,2,3,3,3\}$. This is because if $1$ appears in a multiset 
$A$ any positive number of times, the tree $T_A$ will have a single leaf on one side of the root and 
four leaves on the other side, regardless of what other parameters appear in the set. 
So $A_1$ and $A_2$ are the only elements of $M_5^T$ that we sum over so
\[
p_{T,t}(Bal_5) =\binom{5}{3,2}t_2^3 t_3^2 + \binom{5}{2,3}t_2^2 t_3^3
\]
\end{ex}

\begin{figure}
    \centering
    \begin{subfigure}[b]{0.3\linewidth}
        \centering
        \begin{tikzpicture} [scale = .5,thick]
            \draw [thick] (2,2)--(0,0);
            \draw [thick] (2,2)--(4,0);
            \draw (0,0) node[below]{};
            \draw (4,0) node[below]{};
        \end{tikzpicture}
        \caption{$T$}
        \label{}
    \end{subfigure}
     \begin{subfigure}[b]{0.3\linewidth}
        \centering
        \begin{tikzpicture} [scale = .5,thick]
            \draw [thick] (2,2)--(0,4);
            \draw [thick] (2,2)--(0,0);
            \draw [thick] (2,2)--(4,0);
            \draw (0,0) node[below]{$t_2$};
            \draw (4,0) node[below]{$t_3$};
            \draw (0,4) node[below]{$t_1$};
        \end{tikzpicture}
        \caption{$\tilde{T}$}
        \label{fig:tripod}
    \end{subfigure}
    \begin{subfigure}[b]{0.3\linewidth}
        \centering
        \begin{tikzpicture} [scale = .5,thick]
            \draw (14.5,5)--(12,0);
            \draw (14.5,5)--(17,0);
            \draw (12.5,1)--(13,0);
            \draw (16.5,1)--(16,0);
            \draw (16,2)--(15,0);
        \end{tikzpicture}
        \caption{$Bal_5$}
        \label{fig:balance}
    \end{subfigure}
    \caption{}
    \label{fig:my_label}
\end{figure}

The multinomial model gives a family of distributions as we let the parameter vector $t$ 
range over the entire simplex. Equivalently, the model can be described 
as the image of the simplex under the polynomial map
\[
p_T: \Delta_{|E(T)| - 1} \to EX_n^\infty
\]
where the coordinate corresponding to $S \in EX_n^\infty$ 
has value $p_{T,t}(S)$ for $t \in \Delta_{2m-2}$. Since $\Delta_{2m-2}$ 
is a semialgebraic set and $p_T$ is a polynomial map, 
the multinomial model is also a semialgebraic set. 

It also holds that if we take any tree $T \in RB_U(m)$, and any subtree $T' \in RB_U(m')$ of $T$, 
then we have that $Im(p_{T'}) \subseteq Im(p_T)$. This is because if the 
parameters corresponding to edges that appear in $T$ but not in $T'$ are set to 
$0$ in $p_T$, the map will simply become $p_{T'}$. Setting these parameters to $0$ 
just corresponds to restricting $p_T$ to a subset of the simplex and thus we get the image containment. 

A last interesting note is that this model is perhaps similar in spirit to the $W$-random graphs when $W$ is a graphon obtained from a finite graph $G$ as described in \cite{lovasz12}. The construction begins with a finite graph $G$ and uses it to define a distribution on graphs with $k$ vertices similarly to how we begin with a tree $T$ and define a distribution on trees with $k$ leaves.

We end this section with Figure \ref{fig:ex5withmultinomial},  
which shows both the beta-splitting model and the multinomial model inside $EX_5$.   
In the next section we will discuss the exchangeable and sampling 
consistent distributions on four leaf trees and how they relate 
to the models discussed in this section. 

\begin{figure}
\centering
\includegraphics[scale=.25]{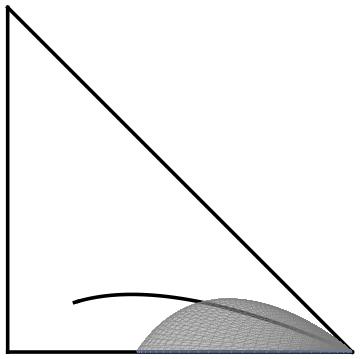}
\caption{This is a projection onto the first two coordinates of the simplex $EX_5$. The beta-splitting model on $RB_U(5)$ is pictured in black and the multinomial model on the two leaf tree is pictured in gray.} 
\label{fig:ex5withmultinomial}
\end{figure}


\section{Distributions in $EX_4^\infty$}
\label{sec:ex4}

In this section we classify all of the distributions in $EX_4^\infty$.
In particular, we show that $EX_4^\infty$ is equal to the beta-splitting model.
 
First we note that since there are only two distinct tree shapes with 
four leaves (see Figure \ref{}), the set of exchangeable distributions 
is just a 1-dimensional simplex $\Delta_1$ in $\rr^2 $.
We take coordinates $(p_1,p_2)$ on $\rr^2$ and let the first coordinate correspond 
to $Comb_4$ and the second coordinate to $Bal_4$. The subset of distributions 
that are also sampling consistent must be some line segment within the simplex. 
We know from Lemma \ref{lem:combvert} that the comb distribution, 
which is $(1,0)$ in these coordinates, 
is a vertex in $EX_4^\infty$. If we can bound the probability of 
obtaining $Bal_4$ then we will have a complete characterization of 
all distributions in $EX_{4}^{\infty}$.  Theorem 14 in \cite{co18}
will be the main tool to achieve this.
 
\begin{thm}\cite[Thm 14]{co18}  \label{thm:Coronado}
The most balanced tree in $RB_U(n)$ has the complete symmetric 
tree on four leaves appear more frequently as a subtree than any other tree in $RB_U(n)$. 
\end{thm}

By the \emph{most balanced tree} in $RB_U(n)$, 
we mean the unique tree shape in $RB_U(n)$ that has the
property that for any internal vertex of the tree, the number of leaves on the left and
right subtrees below that differ by at most one.  

\begin{thm}\label{thm:4leafbeta}
The four leaf beta-splitting model equals the set of all 
exchangeable and sampling consistent distributions on $RB_U(4)$.
\end{thm}

\begin{proof}
Note that  $EX_4^n$ only has two vertices since it is a line segment. 
The comb distribution $(1,0)$ is always a vertex in $EX_4^n$,
by Lemma \ref{lem:combvert}.  The other vertex will be the 
projection of 
the vertex of $EX_n$ that places the most mass on $Bal_4$. 
The projection of a vertex $p_T \in EX_n$, is 
$(p_1,p_2) = \frac{1}{\binom{n}{4}}(m_1,m_2)$ where $m_1$ is the number of $4$ element subsets
$S \subset [n]$ such that $T|_{S} = Comb_4$ and $m_2$ is the number of $4$ element subsets
$S \subset [n]$ such that $T|_{S} = Bal_4$.  By Theorem \ref{thm:Coronado} 
we can restrict to the most balanced tree in $RB_U(n)$.   
We will use $m_{2,n}$ to denote this highest value of $m_2$ that we get 
from the most balanced tree in $RB_U(n)$. 

The beta-splitting model on $RB_U(4)$, on the other hand, is the line segment from 
$(1,0)$ to $(\frac{4}{7},\frac{3}{7})$. 
Indeed, under the beta splitting model, the probability of $Bal_4$ is just
\begin{eqnarray*}
q_4(2)  & =  &   \frac{\binom{4}{2} (\beta + 2)_2^2  }{ (2 \beta + 5)_4 - 2( \beta + 4)_4 }  \\
      &  =  &   \frac{6  \beta^4 + O(\beta^3) }  {  14 \beta^4 + O(\beta^3)}.
\end{eqnarray*}
As $\beta \rightarrow \infty$, this converges to $\tfrac{3}{7}$.
So if we can 
show that $\lim_{n \rightarrow} \tfrac{m_{2,n}}{\binom{n}{4}}  = \tfrac{3}{7}$ then we will be done. 

To prove that $\lim_{n \rightarrow} \tfrac{m_{2,n}}{\binom{n}{4}}  = \tfrac{3}{7}$, 
we can restrict to the subsequence of values $n = 2^k$, since
Lemma \ref{lem:monotone} implies that $\tfrac{m_{2,n}}{\binom{n}{4}}$ is a monotone 
decreasing sequence.  
This subsequence is easier to deal with since $m_{2,2^n}$ counts the number of 
$4$-subsets, $S \subset [2^n]$ of the leaves of the complete symmetric tree 
$T_{2^n}$ in $RB_U(2^n)$ such that $T_{2^n}|_{S} = Bal_4$. 
It is not hard to come up with a simple recurrence for this though since $T_{2^n}$ has 
the recursive structure as illustrated in Figure \ref{fig:t2nbalanced}.

\begin{figure}
    \centering
    \begin{tikzpicture}
        \draw [thick] (1,1)--(0,0);
        \draw [thick] (1,1)--(2,0);
        \draw (0,0) node[below]{$T_{2^{n-1}}$};
        \draw (2,0) node[below]{$T_{2^{n-1}}$};
    \end{tikzpicture}
    \caption{}
    \label{fig:t2nbalanced}
\end{figure}

Note that $m_{2,2^n} = 2m_{2,2^{n-1}} + \binom{2^{n-1}}{2}^2$ since the only 
ways we can choose a subset $S$ such that $T_{2^n}|_S = Bal_4$ are that the leaves 
in $S$ fall either entirely within the left or right subtrees or that $S$ has two leaves 
from both the left and right subtrees. 
The number of ways to choose a subset $S$ that falls entirely 
on the left or right side is  $m_{2,2^{n-1}}$ by definition. 
The number of ways to choose two leaves from each side is  $\binom{2^{n-1}}{2}^2$. 
This recurrence can be solved to find an explicit formula for $m_{2,2^n}$ which is
\[
m_{2,2^n} = \sum_{i=1}^{n-1} 2^{n-i-1}\binom{2^i}{2}^2 
\]
Now we can simplify  $\frac{m_{2,2^n}}{\binom{2^n}{4}}$ to get
\[\frac{m_{2,2^n}}{\binom{2^n}{4}} = \frac{3(2^n) - 5}{7(2^n)-21}\]
which converges to $\frac{3}{7}$ as $n$ tends to infinity. 
\end{proof}

Note that  Theorem \ref{thm:4leafbeta} does not generalize to higher dimensions
as the set of beta splitting distributions is of strictly smaller dimension than
the set of exchangeable sampling consistent distributions.  We explore the
discrepancy between these sets in more detail in the next sections.


\section{Distributions on $EX_5^\infty$}
There are three distinct tree shapes with five leaves so $EX_5$ is a 
$2$-dimensional simplex in $\rr^3$. For the rest of this section we will use 
$Comb_5$, $Gir_5$, and $Bal_5$ to represent the trees pictured in Figure \ref{fig:5leaftrees}. 
Specifically, let $Comb_5$  denote the comb tree on five leaves, $Bal_5$ denote the
balanced tree on five leaves and $Gir_5$ denote the giraffe tree on five leaves. 
We take coordinates $(p_1,p_2,p_3)$ on $\rr^3$ where $p_1,p_2,p_3$  
represent the probability of obtaining $Comb_5$, $Gir_5$, and $Bal_5$, respectively. 

\begin{figure}
    \centering
    
    \begin{tikzpicture} [scale = .5,thick]

    \draw (2.5,5)--(5,0);
    \draw (2.5,5)--(0,0);
    \draw (3,4)--(1,0);
    \draw (3.5,3)--(2,0);
    \draw (4,2)--(3,0);
    \draw (2.5,-.5) node[below]{$Comb_5$};

    \draw (8.5,5)--(6,0);
    \draw (8.5,5)--(11,0);
    \draw (9,4)--(7,0);
    \draw (7.5,1)--(8,0);
    \draw (10.5,1)--(10,0);
    \draw (8.5,-.5) node[below]{$Gir_5$};

    \draw (14.5,5)--(12,0);
    \draw (14.5,5)--(17,0);
    \draw (12.5,1)--(13,0);
    \draw (16.5,1)--(16,0);
    \draw (16,2)--(15,0);
    \draw (14.5,-.5) node[below]{$Bal_5$};
    
    \end{tikzpicture}

    \caption{Tree shapes on five leaves}
    \label{fig:5leaftrees}
\end{figure}
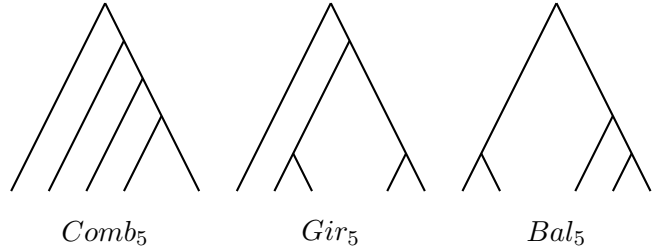

While have not been able to give a complete description of the vertices of 
$EX_5^n$ for all $n$,  we are able to define some tree structures in 
$RB_U(n)$ that do yield vertices of $EX_5^n$. We have already seen that the
comb tree $Comb_m$ always yields a vertex of $EX_n^m$ for all $m$ and $n$.
Here we provide some other examples.

\begin{defn}
For a tree $T \in RB_U(m)$ let $comb(T, n)$  be the tree that is 
obtained by creating a comb tree with $n$ leaves and replacing one of the two leaves 
at the deepest level with the tree $T$.  
\end{defn}

Generally, if $T \in RB_U(m)$ then $comb(T, n)$ has $m + n - 1$ vertices.
For example,  $Gir_5 = comb(Bal_4, 2)$.  
Note that does not matter which of the leaves is replaced with 
$T$ since our trees are unlabelled.

\begin{prop}
Let $T_n = comb(Gir_5,n-4)$.  Then $\pi_5( p_{T_n})$ is a vertex in $EX_{5}^n$.
\end{prop}

\begin{proof}
First note that $T_n$ and $Comb_n$  are the only trees with $n$ leaves
that do not have $Bal_5$ as a subtree. 
This means that $T_n$ and the comb tree fall on the line $p_3 = 0$ in $EX_5$. 
Thus, the set $\{(p_1, p_2, 0) \in EX_n \} \cap EX_5^n$ is a face of $EX_n^n$ for all $ n \geq 5$, 
since every distribution $p \in EX_5$ must satisfy the condition $p_1 + p_2 + p_3 = 1$ 
and thus $p_1 + p_2 \leq 1$. Since $p_1 + p_2 = 1$ is the same line as $p_3 = 0$ it defines a face. 
Now since $\pi_5( p_{T_n})$ and $\pi_5(p_{Comb_n})$ are different points are the only 
distributions of the form $\pi_5(p_T)$ in this face, 
they must be vertices of this face and thus vertices of $EX_5^n$. 
\end{proof}

We now introduce another tree structure that will yield a vertex in $EX_5^n$. 

\begin{defn}
For two positive integers $m$ and $n$ let $bicomb(m,n)$ denote the tree made 
by joining a comb tree of size $m$ and a comb tree of size $n$ together 
at a new root. We call such trees \emph{bicomb trees}.
\end{defn}

For example, $Bal_5 =  bicomb(2,3)$.

\begin{lemma}
Let $T_n = bicomb(\floor{\frac{n}{2}}, \lceil \frac{n}{2} \rceil)$.  Then
$\pi_5( p_{T_n})$ is a vertex of $EX_5^n$. 
\end{lemma}
\begin{proof}
First note that for $n \geq 5$, the only trees in $RB_U(n)$ that never 
contain $Gir_5$ as a restriction tree are the comb tree and the bicomb trees. 
This means that in $EX_5^n$, they are the only trees that fall on the edge $p_2 = 0$. 
To show that $\pi_5( p_{T_n})$ is a vertex of $EX_5^n$  
it remains to to show that $\pi_5( p_{T_n})$ is extremal on this edge. 
We know that the comb tree is one of the extremal points on this edge and 
so the other extremal point will correspond to the bicomb tree with the 
highest density of $Bal_5$ as a restriction tree. 
Let $T' = bicomb(i,n-i)$  be a bicomb tree
for some $1 \leq i \leq n-1$.  We let $b_5(T')$ denote the number 
of times that $Bal_5$ occurs as a restriction tree of $T'$. 
From the structure of a bicomb tree we have
\[
b_5(T') = \binom{i}{2}\binom{n-i}{3} + \binom{i}{3}\binom{n-i}{2}.
\]
This function is maximized when $i = \floor{\frac{n}{2}}$. 
\end{proof}

Now we will show that
the projection of the most balanced tree in $RB_U(n)$ is a 
vertex of $EX_5^n$.   To do this, we prove a few lemmas about
the number of $Comb_5$ trees that can appear as subtrees of a tree.
These results follow the basic outline of Lemmas 12 and 13 
in \cite{co18}, and are in some sense an extension of those results to
$5$ leaf trees.  

For a tree $T \in RB_U(n)$ let 
$c_5(T)$ count the number of $5$-subsets, $S$, of the leaves of $T$ such that $T|_S = Comb_5$. Let $c_4(T)$ be defined similarly, but for $4$ leaf comb trees. 

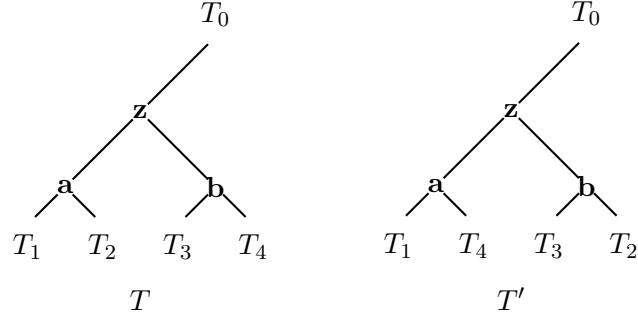
\begin{figure}
    \centering
    \begin{tikzpicture} [scale = .5]
        \node at (5,5) {\textbf{z}};
        \node at (3,3) {\textbf{a}};
        \node at (7,3) {\textbf{b}};
        \node [above] at (7,7) {$T_0$};
        \draw [thick] (5.2,5.2) -- (6.8,6.8);
        \draw [thick] (4.8,4.8) -- (3.2,3.2);
        \draw [thick] (5.2,4.8) -- (6.8,3.2);
        \draw [thick] (2.8,2.8) -- (2.2,2.2);
        \draw [thick] (3.2,2.8) -- (3.8,2.2);
        \draw [thick] (6.8,2.8) -- (6.2,2.2);
        \draw [thick] (7.2,2.8) -- (7.8,2.2);
        \node [below] at (2,2) {$T_1$};
        \node [below] at (4,2) {$T_2$};
        \node [below] at (6,2) {$T_3$};
        \node [below] at (8,2) {$T_4$};
        \node at (5,0) {$T$};
    \end{tikzpicture}  \quad \quad  \quad
    \begin{tikzpicture} [scale =.5]
        \node at (5,5) {\textbf{z}};
        \node at (3,3) {\textbf{a}};
        \node at (7,3) {\textbf{b}};
        \node [above] at (7,7) {$T_0$};
        \draw [thick] (5.2,5.2) -- (6.8,6.8);
        \draw [thick] (4.8,4.8) -- (3.2,3.2);
        \draw [thick] (5.2,4.8) -- (6.8,3.2);
        \draw [thick] (2.8,2.8) -- (2.2,2.2);
        \draw [thick] (3.2,2.8) -- (3.8,2.2);
        \draw [thick] (6.8,2.8) -- (6.2,2.2);
        \draw [thick] (7.2,2.8) -- (7.8,2.2);
        \node [below] at (2,2) {$T_1$};
        \node [below] at (4,2) {$T_4$};
        \node [below] at (6,2) {$T_3$};
        \node [below] at (8,2) {$T_2$};
        \node at (5,0) {$T'$};
    \end{tikzpicture}
    \caption{The two trees from the Proof of  Lemma \ref{lem:switch}.  Note that $T_0$ denotes
    all of the part of the tree that lies above the vertex $z$. }
    \label{fig:switch}
\end{figure}

\begin{lemma}\label{lem:switch}
Let $T$ be as it is pictured in Figure \ref{fig:switch} 
and $T'$ obtained from $T$ by swapping the positions of $T_2$ and $T_3$. For $i=0,1,2,3,4$, let $n_i = \# L(T_i)$ and without loss of generality choose $n_1 \geq n_2$ and $n_3 \geq n_4$. 
If $n_1 > n_3$ and $n_2 > n_4$ then $c_5(T) \geq c_5(T')$.  Furthermore, if $n \geq 7$, then
$c_5(T) > c_5(T')$.
\end{lemma}

\begin{proof}Without loss of generality assume that $n_1 \geq n_2$ and $n_3 \geq n_4$ and let $\Sigma_z$ denote the set of leaves of $T$ below the vertex $z$.  Note that by
construction, this is the same as the set of leaves below the vertex $z$ in $T'$.
If we take a $5$-subset, $S$, of the leaves of $T$ and $T'$ 
then it is only possible for $T|_S \neq T'|_S$ if $|S \cap \Sigma_z| \geq 4$. It is straightforward to see that if $S \cap \Sigma_z$ has zero, one, two, or three elements, $T|_S = T'|_S$.

This means
\[
c_5(T) - c_5(T') = (c_5(T_z) - c_5(T'_z)) + n_0(c_4(T_z) - c_4(T'_z))  
\]
where $T_z$ and $T'_z$ denote the subtrees of $T$ and $T'$ below $z$.
Note that for any tree $S \in RB_U(n)$, it holds that
\[
\binom{n}{4} = c_4(S) + b_4(S)
\]
which gives
\[
n_0(c_4(T_z) - c_4(T'_z)) = n_0(b_4(T'_z) - b_4(T_z))
\]
and $(b_4(T'_z) - b_4(T_z))$ is guaranteed to be positive by Lemma 12 of \cite{co18} so the term $n_0(b_4(T'_z) - b_4(T_z))$ is nonnegative. It remains to show that $(c_5(T_z) - c_5(T'_z))$ is nonnegative. We can explicitly enumerate these quantities in the following way:

\begin{eqnarray*}
c_5(T_z) &=    &\sum_{i=1}^{4}c_5(T_i)+ 
\sum_{i=1}^{4}c_4(T_i)\sum_{j=1,j\neq i}^{4}n_i + \binom{n_1}{3}n_2(n_3+n_4)
\\
&   & \quad \quad \quad \quad \, \, 
+ \binom{n_2}{3}n_1(n_3+n_4) + \binom{n_3}{3}n_4(n_1+n_2)
+ \binom{n_4}{3}n_3(n_1+n_2)
\end{eqnarray*}
\begin{eqnarray*}
c_5(T'_z) &=    &\sum_{i=1}^{4}c_5(T_i)+ 
\sum_{i=1}^{4}c_4(T_i)\sum_{j=1,j\neq i}^{4}n_i + \binom{n_1}{3}n_4(n_2+n_3)
\\
&   & \quad \quad \quad \quad \, \, 
+ \binom{n_4}{3}n_1(n_2+n_3) + \binom{n_2}{3}n_3(n_1+n_4)
+ \binom{n_3}{3}n_2(n_1+n_4)
\end{eqnarray*}
We can simplify this to get that
\[
c_5(T_z) - c_5(T'_z) = \frac{1}{6} (n_1 - n_3) (n_2 - n_4) (n_1 n_3 (-3 + n_1 + n_3) + n_2 n_4 (-3 + n_2 + n_4)).
\]
Note that this quantity is greater than $0$ since $n_1 > n_3$ and $n_2 > n_3$ 
by assumption and $n_i \geq 1$ for $i=1,2,3,4$. Note that if $n \geq 7$, then we either have that $n_0 \geq 1$, or $\sum_{i=1}^4 n_i \geq 7$ which both guarantee that $c_5(T) - c_5(T') > 0$.
\end{proof}

This lemma essentially tells us that if the tree has an internal node that is unbalanced, we can find a tree that has $Comb_5$ appear less frequently as a restriction tree. We now have another lemma following in the style of \cite{co18}. 

\begin{lemma}
\label{lem:switch_leaf}
Let $T$ be as it is pictured in Figure 8 and for $i=0,1,2$, let $n_i = \# L(T_i)$ and assume $n_1 \geq n_2$. We also assume that $n_1+n_2 \geq 3$. Then $c_5(T) \geq c_5(T')$.  
Furthermore, if $n \geq 7$, then
$c_5(T) > c_5(T')$.
\end{lemma}
\begin{figure}
    \centering
    \begin{tikzpicture} [scale = .5]
        \node at (5,5) {\textbf{z}};
        \node at (3,3) {\textbf{a}};
        \node [above] at (7,7) {$T_0$};
        \draw [thick] (5.2,5.2) -- (6.8,6.8);
        \draw [thick] (4.8,4.8) -- (3.2,3.2);
        \draw [thick] (5.2,4.8) -- (7.8,2.2);
        \draw [thick] (2.8,2.8) -- (2.2,2.2);
        \draw [thick] (3.2,2.8) -- (3.8,2.2);
        \node [below] at (2,2) {$T_1$};
        \node [below] at (4,2) {$T_2$};
        \node [below] at (8,2) {$l$};
        \node at (5,0) {$T$};
    \end{tikzpicture}
    \begin{tikzpicture} [scale =.5]
        \node [above] at (7,7) {$T_0$};
        \draw [thick] (5.2,5.2) -- (6.8,6.8);
        \node at (5,5) {\textbf{z}};
        \node at (7,3) {\textbf{b}};
        \draw [thick] (4.8,4.8)--(2.2,2.2);
        \node [below] at (2,2) {$T_1$};
        \draw [thick] (5.2,4.8)--(6.8,3.2);
        \draw [thick] (6.8,2.8)--(6.2,2.2);
        \draw [thick] (7.2,2.8)--(7.8,2.2);
        \node [below] at (6,2) {$T_2$};
        \node [below] at (8,2) {$l$};
        \node at (5,0) {$T'$};
    \end{tikzpicture}
    \caption{}
    \label{fig:my_label}
\end{figure}

\begin{proof}
We will again proceed by showing that $c_5(T) - b_5(T') > 0$. By the same reasoning as that given in the last lemma we know that 
\[c_5(T) - c_5(T') = c_5(T_Z) - c_5(T'_z) + n_0(c_4(T_z) - c_4(T'_z))\]
and the nonnegativity of the second term follows in the same manner that was described in the previous lemma. 
Now we can easily see that 
\[c_5(T_z) = c_5(T_1) + c_5(T_2) + (n_2+1)c_4(T_1) + (n_1+1)c_4(T_2) + \binom{n_1}{3}n_2 + \binom{n_2}{3}n_1 
\]
\[c_5(T'_z) = c_5(T_1) + c_5(T_2) + (n_2+1)c_4(T_1) + (n_1+1)c_4(T_2) + \binom{n_2}{3}n_1 \]
and so
\[c_5(T_Z) - c_5(T'_z) = \binom{n_1}{3}n_2 \]
It is clear that the right hand side is always nonnegative. Note that if $n \geq 7$, then either $n_0 \geq 1$ or $n_1 \geq 3$. In both cases this guarantees that $c_5(T) - c_5(T') > 0$. 

\end{proof}

Combining these two lemmas together we get the following theorem. This theorem will immediately allow us to show that the projection of the most balanced tree in $RB_U(n)$ will always be a vertex in $EX_5^n$. 

\begin{thm}
\label{thm:minc5}
For $n \geq 7$, the minimum value of $c_5(T)$ is attained when every internal node of $T$ is 
maximally balanced.
\end{thm}

\begin{proof}
This proof also follows the strategy of \cite{co18}. 
We assume that $c_5$ obtains it minimum value in $RB_U(n)$ at $T$ but that 
$T$ is not maximally balanced. We will try to find a contradiction. 
We let $z$ be a non-balanced internal node with balanced children $a$ and $b$. 
We let $n_a$ and $n_b$ be the number of leaves of the trees rooted at $a$ and $b$ respectively. 
Then since $z$ is not balanced we have, without loss of generality, that 
$n_a \geq n_b +2$. If $b$ is a leaf then by Lemma \ref{lem:switch_leaf} 
we immediately have that $c_5(T)$ is not minimum since 
$n \geq 7$. So we have that $n_b \geq 2$ and thus both 
$a$ and $b$ are balanced and must be internal nodes.

We now let $v_1,v_2$ be the children of $a$ and $v_3,v_4$ 
be the children of $b$ and take $n_i = \#L(T_{v_i})$ for $i=1,2,3,4$ 
and once again without loss of generality assume that $n_1 \geq n_2$ and 
$n_3 \geq n_4$. Since both $a$ and $b$ are balanced it must be that $n_1=n_2$ or 
$n_1 = n_2 +1$ and $n_3 = n_4$ or $n_3 = n_4 +1$. Then the assumption that 
$n_a \geq n_b + 2$ immediately gives us that 
\[n_1 + n_2 = n_a \geq n_b + 2 = n_3 + n_4 +2\]
Then by previous assumptions we get that $n_1 > n_3$. Now since $c_5$ is minimum at $T$ and $n \geq 7$, we can apply Lemma \ref{lem:switch} to get that $n_4 \geq n_2$. Stringing together these inequalities we get that 
\[n_1 > n_3 \geq n_4 \geq n_2\]
But since $n_1 = n_2$ or $n_1 = n_2 +1$, the only possibility we have is that
\[n_1 -1 = n_2 = n_3 = n_4\]
But then we get that $n_1 + n_2 = 2n_1 -1$ and $n_3 + n_4 = 2n_1 -2$
which contradicts the inequality $n_1 + n_2 \geq n_3 + n_4 +2$. 
This tells us that any tree with at least $7$ leaves must be maximally 
balanced around every internal node if it obtains the minimum value of 
$c_5$ on $RB_U(n)$. Since there is only one tree that is maximally balanced 
at every internal node, there is a unique minimizer of $T$ in $RB_U(n)$ for $n \geq 7$. 
\end{proof}

\begin{cor}
Let $T_n$ be the maximally balanced tree in $RB_U(n)$. Then $\pi_5(T_n)$ is a vertex of $EX_5^n$. 
\end{cor}

\begin{proof}
The Corollary can be verified computationally for $n = 6$.  For $n \geq 7$
Theorem \ref{thm:minc5} shows that $T_n$ is the unique tree that attains 
the minimum value of $c_5$ among all trees in $RB_U(n)$. So it holds that 
$\{(p_1,p_2,c_5(T_n)) \in EX_5 \} \cap EX_5^n = \{\pi_5(T_n)\}$, thus $\pi_5(T_n)$ is a vertex of $EX_5^n$.
\end{proof}

We have another Corollary that relates the exchangeable and 
sampling consistent distributions to the $\beta$-splitting model.

\begin{cor}
The projection of the most balanced tree in $EX_5^n$ approaches the $\beta=\infty$ point on the beta-splitting model as $n \to \infty$.
\end{cor}
 
\begin{proof}
It is enough to show that the complete symmetric tree $T_{2^n} \in RB_U(2^n)$ 
satisfies this property. We can just count the number of times that $Gir_5$ and 
$Bal_5$ occur as restriction trees when we restrict to a 5-subset of the leaves.
We will call these quantities $g_{5,2^n}$ and $b_{5,2^n}$ respectively. Once again 
since $T_{2^n}$ has the structure depicted in Figure 5 and we can use this structure 
to write down a simple recurrence for $g_{5,2^n}$ and $b_{5,2^n}$ and then solve the recurrence. 
Since we can either choose our subset to be on either the right or 
left side of the tree or 3 leaves from one side and 2 leaves from the other, $b_5$ is simply
\[
b_5(T_{2^n}) = 2b_5(T_{2^{n-1}}) + 2\binom{2^{n-1}}{3}\binom{2^{n-1}}{2}
\]
As for $g_5$, we can once again choose our subset to be on either the right or left side of the tree or we can choose to have 1 leaf on a side of the tree and a 4 leaf symmetric tree on the other. This can be done in just $2^{n-1}m_{2,2^{n-1}}$ ways. So $g_5(T_{2^n})$ is just
\[g_5(T_{2^n}) =  2g_5(T_{2^{n-1}})+ 2(2^{n-1}m_{2,2^{n-1}}) = 2g_5(T_{2^{n-1}}) + 2^{n}m_{2,2^{n-1}}\]
Both of these recurrences can be solved explicitly using a computer algebra system. We get that
\[b_5(T_{2^n}) = \frac{1}{315}2^{n-2}(2^n-4)(2^n-2)(2^n-1)(7*2^n-11)\]
\[g_5(T_{2^n}) = \frac{1}{105}2^{n-3}(2^n-4)(2^n-3)(2^n-2)(2^n-1)\]
We can then find the probabilities $p_2$ and $p_3$ of $Gir_5$ and $Bal_5$ by simply dividing out by $\binom{2^n}{5}$. This yields
\[p_3 = \frac{b_5(T_{2^n})}{\binom{2^n}{5}} = \frac{2}{3} + \frac{20}{21(2^n-3)}\]
\[p_2 = \frac{g_5(T_{2^n})}{\binom{2^n}{5}} = \frac{1}{7}\]
Clearly as  $n \rightarrow \infty$ we have $p_3 \rightarrow \frac{2}{3}$ and 
$p_2 \rightarrow \frac{1}{7}$. 

On the other hand, we recall that the probability of obtaining a tree 
under the beta-splitting model is just a rational function in $\beta$ that can
be explicitly calculated. We can then find the limit of these rational 
functions to get that the beta-splitting curve approaches the point
\[(p_1,p_2,p_3) = (\frac{4}{21},\frac{1}{7},\frac{2}{3}) \]
as $\beta \to \infty$ as well and so the projection of $T_{2^n}$ in 
$EX_5^{2^n}$ is approaching the $\beta=\infty$ point on the curve. 
\end{proof}

These are all of the tree structures in $RB_U(n)$ we have been able to 
find that always appear as vertices in $EX_5^n$. We end this section 
with Figure \ref{fig:alldists}, which pictures all of the families of 
exchangeable and sampling consistent distributions that we have 
discussed and the vertices of $EX_n^m$ for some small values of $m$. 

\begin{figure}
    \centering
    \includegraphics[scale=.5]{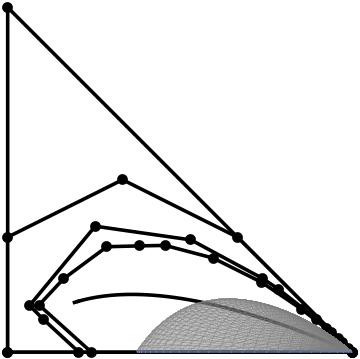}
    \caption{The multinomial model on the three leaf comb tree is in grey and the $\beta$-splitting model is the thick black curve. The thinner black lines are the boundary of $EX_5^n$ for n=5,6,9,12.}
    \label{fig:alldists}
\end{figure}


\section{Distributions on $EX_n^\infty$}
While we are not able to get a description of the vertices of $EX_n^m$ 
for general $m$ and $n$, it is possible to to describe $EX_n^\infty$ 
using the multinomial model that was introduced in Section \ref{sec:multimodel}. 
In particular, this shows that multinomial models converge as an inner
limit to $EX_n^\infty$.  

\begin{thm}
\label{thm:indmulti}
Let $\{T_m\}_{m=n}^\infty$ be a sequence of tree shapes and $p^{(m)} = \pi_n(T_m)$
be the corresponding sequence of distributions. If $p^{(m)}$ converges to some 
$p \in EX_n^\infty$ as $m$ goes to infinity, then there exists a sequence of 
multinomial distributions $\{d^{(m)}\}_{m=n}^\infty$ that also converges to $p$ as $m$ goes to infinity. 
\end{thm}

\begin{proof}
Define $d^{(m)}$ to be the multinomial distribution on the 
tree $T_m$ with the edge parameter vector $(t_e |e \in E(T_m))$ such that 
$t_e = \frac{1}{m}$ if one of the vertices in $e$ is one of the original 
$m$ leaves of $T_m$ and $t_e = 0$ otherwise. Note that these nonzero edge 
parameters are bijectively associated to the leaves of $T_m$ and we may 
call the set of nonzero edge parameters $L(T_m)$ meaning the leaf set of $T_m$. 
To show that $d^{(m)}$ also converges to $p$, it is enough to show that for 
every tree $T \in RB_U(n)$, $\lim_{m \to \infty}d^{(m)}(T) = \lim_{m \to \infty}p^{(m)}(T)$. 
Fix a labelling of $T_m$ and let $c_{T_m}(T)$ be the number of sets 
$S \subseteq [m]$ such that $shape(T_m|_S) = T$. By Corollary \ref{cor:induceddensity}, 
$p^{(m)}(T)$ is the induced subtree density of $T$ in $T_m$, 
so $p^{(m)}(T) = \frac{c_{T_m}(T)}{\binom{m}{n}}$. So
\[
\lim_{m \to \infty}p^{(m)}(T) = \lim_{m \to \infty} \frac{c_{T_m}(T)}{\binom{m}{n}} 
    = \lim_{m \to \infty} \frac{n!}{m^n}c_{T_m}(T)
\]

On the other hand, let $M^{(m)} = \{A \in M_n^{T_m} |{T_m}_A = T, poly(A) \neq 0\}$, then
\[
d^{(m)}(T) = \sum_{A \in M^{(m)}} poly(A) 
\]
by definition and we note by requiring that multisets $A \in M^{(m)}$ have that $poly(A) \neq 0$, $M^{(m)}$ only includes multisets whose support is contained in $L(T_m)$. Also note that $poly(A)$ is either $0$ or $\binom{n}{m_A(t_{e_1}),m_A(t_{e_2}), \ldots m_A(t_{e_{2m-1}})}\frac{1}{m^n}$ since all the edge parameters are $0$ or $\frac{1}{m}$. 
So to understand the quantity $d^{(m)}(T)$ it is enough to know the coefficient of $\frac{1}{m^n}$. Note that any multiset $A$ has a naturally associated integer partition of $n$ to it, formed by taking the multiplicities of each unique element that appears in it. Call this integer partition the weight of $A$, denoted $wt(A)$, and let $M_\lambda^{(m)}$ be the set of multisets in $M^{(m)}$ with weight $\lambda$. Now observe that for $A,B \in M_\lambda^{(m)}$, $poly(A) = poly(B)$ since the value of the multinomial coefficient is totally determined by the weight and the product of the edge parameters is always $\frac{1}{m^n}$. If we let $\binom{n}{\lambda}$ be the value of the multinomial coefficient then the formula for $d^{(m)}(T)$ can be rewritten as
\[
d^{(m)}(T) = \frac{1}{m^n}\sum_{\lambda \vdash n} \binom{n}{\lambda}|M_\lambda^{(m)}|
\]
but we can bound the quantity $|M_\lambda^{(m)}|$. We note that the quantity $|(M_n^{T_m})_\lambda|$, of all multisets on the edge parameters of $T_m$ of size, with weight $\lambda$, is at most $l(\lambda)!\binom{m}{l(\lambda)}$ where $l(\lambda)$ is the length of the partition $\lambda$. This is because there are $\binom{m}{l(\lambda)}$ choices for which elements to use in the multiset and at most $l(\lambda)!$ unique multisets for each choice of elements. Since $l(\lambda)!\binom{m}{l(\lambda)}$ is a polynomial in $m$ of degree $l(\lambda)$ though, we have that
\[
\lim_{m \to \infty}  \frac{1}{m^n}\sum_{\lambda \vdash n} \binom{n}{\lambda}|M_\lambda^{(m)}| = 
\lim_{m \to \infty} \frac{n!}{m^n}|M_{(1,1,\ldots, 1)}^{(m)}|
\]
since the partition $\lambda = (1,1,\ldots 1)$ is the only partition where $|M_{(1,1,\ldots, 1)}^{(m)}|$ is of the order $m^n$, and so is the only term that contributes to the limit. Now we note that the multisets $A \in M_{(1,1,\ldots, 1)}^{(m)}$ correspond exactly to choosing subsets of the leaves of $T_m$ that yield $T$ upon restriction since the only edges that can be in $A$ are those corresponding to leaves, every leaf can be chosen at most once, and $shape({T_m}_A) = T$. So $|M_{(1,1,\ldots, 1)}^{(m)}| = c_{T_m}(T)$, and so
\[
\lim_{m \to \infty}d^{(m)} = \lim_{m \to \infty} \frac{n!}{m^n}c_{T_m}(T) = \lim_{m \to \infty}p^{(m)}
\]
and since $p^{(m)}$ converges, to $p$, it must be that $d^{(m)}$ also does.
\end{proof}

\begin{cor}
\label{cor:approxIndMulti}
Suppose that $p \in EX_n^m$ for some $m > n$. Then for any tree $S \in RB_U(n)$, $p(S)$ can be approximated with a distribution $d \in EX_n^\infty$ with error $\frac{C}{m}$, where $C$ is a constant with respect to $m$ that does not depend on the tree $S$. 
\end{cor}
\begin{proof}
Note that if $p \in EX_n^m$, then we have for every $S \in RB_U(n)$,
\[
p(S) = \sum_{T \in RB_U(m)} \lambda_T \pi_n(p_T)(S)
\]
where the above combination is convex by Lemma \ref{lemma:exnm}. Then let $d^T$ be defined as the multinomial distribution $d^T$ on $T$ just as $d^{(m)}$ is defined for $T_m$ in the previous theorem. Then recall from the proof of the previous theorem that
\[
d^T(S) = \frac{1}{m^n}\sum_{\lambda \vdash n} \binom{n}{\lambda}|M_\lambda^T|
\]
where $M_\lambda^T = \{A \in M_n^T |T_A = S, ~poly(A) \neq 0, ~wt(A) = \lambda \}$. Also recall from the proof of the previous theorem that $|M_{(1,1,\ldots,1)}^T| = c_T(S)$. Combining these facts with the definition of $\pi_n(p_T)$ and the triangle inequality gives 
\begin{equation}
\label{eq:tri}
    |\pi_n(p_T)(S)- d^T(S)| \leq \bigg|\frac{c_T(S)}{\binom{m}{n}} - \frac{n!c_T(S)}{m^n}\bigg| + \bigg|\frac{1}{m^n} \sum_{\substack{\lambda \vdash n \\ \lambda \neq (1,1,\ldots,1)}} \binom{n}{\lambda}|M_\lambda^T| \bigg|
\end{equation}
and we now bound each term on the right hand side of this inequality. 

To bound the first term in equation (\ref{eq:tri}), note that $c_T(S)$ is a nonnegative quantity and is bounded above by $\binom{m}{n}$. This gives the inequality
\begin{equation}
\label{eq:firstterm}
    \bigg|\frac{c_T(S)}{\binom{m}{n}} - \frac{n!c_T(S)}{m^n}\bigg| \leq
    \bigg| 1-\frac{\frac{m!}{(m-n)!}}{m^n} \bigg| \leq \bigg|\frac{m^n - (m-n)^n}{m^n} \bigg| \leq \frac{C_1}{m}
\end{equation}
where $C_1 \in \mathbb{R}$ is a constant. Note that this constant does not depend on the trees $T$ and $S$.

To bound the second term we again recall from the proof of the previous theorem that $|M_\lambda^T| \leq l(\lambda)!\binom{m}{l(\lambda)}$ for each partition $\lambda$ of $n$. Then we have that
\begin{equation}
\label{eq:secondterm1}
\bigg|\frac{1}{m^n} \sum_{\substack{\lambda \vdash n \\ \lambda \neq (1,1,\ldots,1)}} \binom{n}{\lambda}|M_\lambda^T| \bigg| \leq
\sum_{\substack{\lambda \vdash n \\ \lambda \neq (1,1,\ldots,1)}} \binom{n}{\lambda}\frac{l(\lambda)! \binom{m}{l(\lambda)}}{m^n}
\end{equation}
but since $\lambda \neq (1,1,\ldots,1)$, it must be that $l(\lambda) \leq n-1$ so $l(\lambda)!\binom{m}{l(\lambda)} \leq m^{n-1}$ for all the remaining partitions $\lambda$. Applying this fact to the right hand side of equation (\ref{eq:secondterm1}) gives the bound
\begin{equation}
\label{eq:secondterm2}
    \bigg|\frac{1}{m^n} \sum_{\substack{\lambda \vdash n \\ \lambda \neq (1,1,\ldots,1)}} \binom{n}{\lambda}|M_\lambda^T| \bigg| \leq
    \frac{1}{m} \sum_{\substack{\lambda \vdash n \\ \lambda \neq (1,1,\ldots,1)}} \binom{n}{\lambda} \leq \frac{C_2}{m}
\end{equation}
where $C_2 \in \mathbb{R}$ is a constant that also does not depend on the trees $T$ and $S$. Applying the bounds for each term to equation (\ref{eq:tri}) and setting $C=C_1 + C_2$ gives 
\begin{equation}
\label{eq:perTreeBound}
    |\pi_n(p_T)(S)- d^T(S)| \leq \frac{C}{m}
\end{equation}
and again we note that $C$ is independent of the trees $T$ and $S$ since $C_1$ and $C_2$ are. We are now ready to construct a distribution $d \in EX_n^\infty$ that gives the desired result. 
From the discussion of the multinomial model, we have that each distribution $d^T \in EX_n^\infty$ and so from the convexity of $EX_n^\infty$ we get
\[
d = \sum_{T \in RB_U(m)}\lambda_T d^T \in EX_n^\infty.
\]
We can now use the expression for $p$ we began with and the bound obtained in equation (\ref{eq:perTreeBound}) to get that 
\[
|p(S) - d(S)| \leq 
\sum_{T \in RB_U(m)} \lambda_T |\pi_n(p_T)(S) - d^T(S)| \leq
\frac{C}{m}.
\]

\end{proof}

Theorem \ref{thm:indmulti} gives that the limit of any convergent sequence 
$(v_m)_{m \geq 1}$ where $v_m \in V(EX_n^m)$ can also be realized 
as the limit of points coming from multinomial models. Corollary \ref{cor:approxIndMulti} shows that if we have a distribution in $EX_n$ that can be extended to part of a finitely sampling consistent family, then it can be approximated with an infinitely sampling consistent distribution. 
With Theorem \ref{thm:indmulti} and the following theorem, we will show that 
$EX_n^\infty$ is actually the convex hull of all limits of 
convergent sequences of vertices, and thus the convex hull of limits 
of distributions drawn from the multinomial model.  To do this
we need a basic proposition from convex analysis which the proof of is included for completeness.  

\begin{prop}
\label{thm:vertlims}
Let $(P_m)_{m \geq 1}$ be a sequence of polytopes in $\mathbb{R}^n$ 
such that for all $m \geq 1$, $P_{m+1} \subseteq P_m$. Let
\[
P =  \overline{ {\rm conv}(
\{\lim_{m \to \infty}  v_{i_m}^{(m)} | v_{i_m}^{(m)} \in V(P_m) 
\mbox{ and } 
(v_{i_m}^{(m)})_{m \geq 1} \mbox{ converges }\} ) }
\]
where the bar denotes the closure in the Euclidean topology.  
Then $P  = \cap_{m=1}^\infty P_m$.
\end{prop}

\begin{proof}
It is straightforward to see that $P \subseteq \cap_{m=1}^\infty P_m$. 
To show that the sets are equal suppose that there is 
$p \in (\cap_{m=1}^\infty P_m ) \setminus P$. Then the Basic Separation Theorem
of convex analysis implies there must exist an affine functional 
$\ell$ with $\ell(p) \leq 0$ and $\ell(w) > 0$ for all $w \in P$. 
We also have that since $p \in \cap_{m=1}^\infty P_m$, for each $m \geq 1$, $p$ can be written as
\[
p = \sum_{j=1}^{k_m} \lambda_j v_j^{(m)}
\]
where the $v_j^{(m)}$ are the vertices of $P_m$. 
Then because $\ell(p) < 0$ it must be that for each $m$, 
there exists at least one vertex $v_{i_m}^{(m)}$ of $P_m$ such that $\ell(v_{i_m}^{(m)}) < 0$. 
Since all the points $v_j^m$  lie in $P_1$ which is a compact set, 
there exists a convergent subsequence 
$(v_{i_{m_k}}^{(m_k)})_{k \geq 1}$ with limit $v \in P$, thus $\ell(v) > 0$. But it also holds that
\[
\ell(v) = \lim_{k \to \infty} \ell(v_{i_{m_k}}^{(m_k)}) \leq 0 
\]
which is a contradiction. 
\end{proof}

\begin{cor}
\label{cor:exmultilims}
Let $d_{T_m}^{(m)}$ denote the specific multinomial model construction on the tree $T_m \in RB_U(m)$ described in Theorem \ref{thm:indmulti}.  
Then
\[
EX_n^\infty = \overline{{\rm conv}(\{\lim_{m \to \infty}d_{T_m}^{(m)}| \pi_n(T_m) \in V(EX_n^m) \mbox{  and } d_{T_m}^{(m)}  \mbox{ converges } \} ).  }
\]
\end{cor}
\begin{proof}
Recall that $EX_n^\infty = \cap_{m = n}^\infty EX_n^m$, thus by Proposition \ref{thm:vertlims}, 
\[
EX_n^\infty = \overline{{\rm conv}(\{\lim_{m \to \infty}  
\pi_n(p_{T_m}) | T_m \in RB_U(m) \mbox{ and  } (\pi_n(p_{T_m}))_{m \geq 1} \mbox{ converges } \}) }
\]
since the vertices of $EX_n^m$ correspond to a subset of the points $\pi_n(T_m)$. Applying Theorem \ref{thm:indmulti} to the sequence $(\pi_n(T_m))_{m \geq 1}$ gives the result. 
\end{proof}

Corollary \ref{cor:exmultilims} shows that every exchangeable and
infinitely sampling consistent distribution is either a convex combinations of limits of multinomial distributions or a limit point of points in that set. Understanding the structure of the multinomial models may shed greater light on the structure of $EX_n^\infty$ as a whole. We view Corollary \ref{cor:approxIndMulti} and
Corollary \ref{cor:exmultilims} as the rooted binary tree analogue to Theorems 3 and 4 in \cite{di97}, in essence they are finite forms of a deFinetti-type theorem for rooted binary trees. As previously mentioned, the work done in \cite{formanPitman18} and \cite{forman18} establishes a more typical deFinetti theorem in the sense that it shows every infinitely sampling consistent sequence of distributions can be obtained by sampling from a limit object using techniques from Probability theory.  

We also note that the requirement that the induced subtree densities converge is quite similar to the idea of graph convergence that appears in \cite{lovasz12} and that many of the ideas in the theory of graph limits may also be applied to trees. The very well developed theory of graph limits contains many equivalent versions of the limiting object (see Theorem 11.52 in \cite{lovasz12}). The work done in \cite{formanPitman18} and \cite{forman18} makes the connection between the limiting object,a random real tree, and an infinitely sampling consistent model. It is still unknown if this can be connected to ideas such as tree parameters (the induced subtree density for instance) and to metrics on finite trees as has been done in the theory of graph limits. It seems that many of these equivalences hold but differences in techniques will be required.

\section*{Acknowledgments}

Benjamin Hollering and Seth Sullivant were partially supported by the US National Science Foundation
(DMS 1615660).  Thanks to D\'avid Papp for a helpful conversation regarding Proposition
\ref{thm:vertlims}.

\nocite{*}
\bibliography{ex_sc-2_8}{}
\bibliographystyle{plain}

\end{document}